\documentclass[preprint,12pt,times]{article}
\usepackage{graphicx}
\usepackage{amsthm}
\usepackage{amssymb}
\usepackage{amsmath}
\usepackage{amscd}
\usepackage{bbm}
\usepackage{float}
\usepackage{booktabs}
\usepackage{tikz}
\usepackage{pgfplots}
\usepackage[margin=2.5cm]{geometry}
\usepackage[labelsep=period]{caption}
\usepackage{adjustbox}
\numberwithin{equation}{section}

\newtheorem{theorem}{Theorem}[section]

\newtheorem{remark}{Remark}[section]

\usepackage[numbers,sort&compress]{natbib}

\numberwithin{figure}{section}
\numberwithin{table}{section}

\allowdisplaybreaks[2]
\newcommand{\ep}{\mathrm{e}}

\newcommand\btd{\raise 2pt \hbox{$\hat\bigtriangledown$}\hskip 1.5pt}
\newcommand\bt{\raise 2pt \hbox{$\bigtriangledown$}\hskip 1.5pt}

\begin{document}
\date{}
\title{Wave breaking for the nonlocal Whitham-type equations}
%

\author{
Jian Chen$^{1}$\thanks{chenj883@mail2.sysu.edu.cn (Jian Chen)},~~
Shaojie Yang$^{2}$\thanks{Corresponding author: shaojieyang@kust.edu.cn (Shaojie Yang)},~~
Zhaoyang Yin$^{1}$\thanks{mcsyzy@mail.sysu.edu.cn (Zhaoyang Yin)}
\\[6mm]
\small $^{1}$ School of Science,  Shenzhen Campus of Sun Yat-sen University,\\
 \small Shenzhen, Guangdong 518107, China
\\~\\
\small $^{2}$ Department of Mathematics,  Faculty of Science, Kunming University of Science and Technology,\\
 \small Kunming,  Yunnan 650500, China
}

\date{}
\maketitle
\begin{abstract}
In this paper, we investigate wave breaking for the nonlocal Whitham-type equations.
We  establish  more general wave breaking criteria that strictly extend earlier wave breaking criteria for the nonlocal Whitham-type equations.  The approach relies on analyzing the  monotonicity and continuity properties of a linear combination comprising the supremum and infimum of the solution's slope.  Further generalizations are demonstrated via applications to the Camassa-Holm equation, the $\mu$-Camassa-Holm equation and the hyperelastic rod equation.  Our refined approach yields more  precise  wave breaking criteria, thereby extending the theoretical understanding of singularity formation in dispersive waves.  \\


\end{abstract}
\noindent\rule{15.5cm}{0.5pt}

\newpage

\section{Introduction}\label{sec1}

In 1834,  Russell observed that during the propagation of water waves, the shape and speed remained largely unchanged, with only a gradual decrease in amplitude until eventual dissipation \cite{r1}. This phenomenon was later called  solitary waves. The celebrated Korteweg-de Vries (KdV) equation
\begin{align}\label{KdV}
\begin{cases}
u_t + 6uu_x + u_{xxx} = 0,\\
u(0,x)=u_0(x),
\end{cases}
\end{align}
was in fact derived formally by Boussinesq \cite{r43} in 1877 and was later derived and studied by Korteweg and de Vries \cite{r5} in 1895. Most remarkably, the KdV equation admits soliton solutions---waves that, after colliding and separating, retain their original shapes and amplitudes. The KdV equation possesses a bi-Hamiltonian structure, and its integrability allows for the application of the inverse scattering method to obtain precise information about its  structure \cite{r6}. The variational study of solitons dates back to the work of Boussinesq \cite{r40}, in which it was asserted that the Hamiltonian is minimized at a solitary-wave profile under a fixed $L^2$ constraint. The first rigorous orbital stability result for KdV solitary waves in $H^1(\mathbb{R})$ was proved by Benjamin \cite{r41} and Bona \cite{r42}.

However, when $u_0\in H^1(\mathbb{R})$, the solution of the KdV equation \eqref{KdV} is global (see \cite{r7}),  and therefore fails to capture wave breaking. Whitham \cite{r8} emphasized the breaking phenomenon is one of the most intriguing long-standing problems of water wave theory.  Since the KdV equation cannot describe wave breaking, Whitham  noted that “it is intriguing to know what kind of simpler mathematical equation (than the governing equations of the water wave problem) could include” breaking effects, and he \cite{r8,r9} subsequently proposed the following nonlocal equation
\begin{align}\label{w}
\begin{cases}
u_t+u u_x+\int_{\mathbb R}K(x-\xi)u_\xi(t,\xi)\,d\xi=0,\\
u(0,x)=u_0(x),
\end{cases}
\end{align}
with the singular kernel
\begin{align*}
K(x)=\frac{1}{2\pi}\int_{\mathbb{R}}\left(\frac{\tanh \xi}{\xi} \right)^{\frac{1}{2}}\ep^{ix\xi}d\xi,
\end{align*}
as a  simple model that combines full linear dispersion with long-wave nonlinearity, and it was conjectured to capture wave breaking. Klein et al. \cite{r44} established rigorous bounds between solutions of the Whitham and KdV equations in the long-wave regime.

Since the integral in Eq.\eqref{w} has a singular kernel and is difficult to handle, Whitham \cite{r9} in 1967 and Whitham and Fornberg \cite{r10} in 1978, proposed the Fornberg-Whitham (FW) equation
\begin{align}\label{fw}
\begin{cases}
u_t+\frac{3}{2}uu_x=G*u_x,\\
u(0,x)=u_0(x),
\end{cases}
\end{align}
where $G(x)=\frac{1}{2}\ep^{-|x|}$ is the kernel of the operator $(1-\partial_x^2)^{-1}$ and
\begin{align*}
G*u_x=\frac{1}{2}\int_{\mathbb{R}}\ep^{-|x-\xi|}u_\xi(t,\xi)d\xi,
\end{align*}
as a comparably simple model alternative to the  Whitham equation \eqref{w} while still  capturing wave breaking.


A fundamental question in the theory of nonlinear PDEs concerns the formation and property of singularities. The standard well-posedness theory (see \cite{r11}) asserts that either PDEs exists globally or develop a singularity at a finite time $T<\infty$ such that some norm of the solution becomes unbounded as $t\to T$. the latter phenomenon is called finite time blow-up.

The asymptotic behavior of the solution in the approach to the blow-up time is of particular interest. A simple singularity emerges when the solution itself becomes unbounded within a finite time. In models of water waves, wave breaking is said to occur when the solution (representing the wave profile) remains bounded, yet its slope becomes unbounded in finite time. The wave steepens progressively as it propagates until a point of vertical slope develops—at which point the wave is said to break \cite{r8}.

The formal approach to proving wave breaking for Whitham-type equations, originating from Seliger’s ingenious argument \cite{r12}, consider
\begin{align}\label{a.2}
m(t):=\underset{x\in \mathbb{R}}{\text{inf}}\{u_x(t,x)\}~~\text{and} ~~M(t):=\underset{x\in\mathbb{R}}{\text{sup}}\{u_x(t,x)\},
\end{align}
attained at $x=\xi_1(t)$ and $x=\xi_2(t)$, respectively,  provided  $K(x)$  be bounded and integrable, among other hypotheses.  However, since the mapping $t\to\xi_i(t)$ may be multi-valued, these curves are not generally well-defined.
 Moreover, Seliger's formal analysis relies on the additional smoothness assumption for the curves $\xi_1(t)$ and $\xi_2(t)$.  Constantin and Escher \cite{r4} later eliminated these restrictive conditions through a rigorous analysis. Motivated by their methodology, Ma, Liu and Qu \cite{r2} derived an improved wave breaking criterion.


In the following,  we recall some earlier results on wave breaking for
the Whitham equation and the FW equation,  based on the methods mentioned in the previous paragraph. For other wave breaking results, we  refer the reader to \cite{r13,r14,r15,r16,r17,r18,r25,r45,r46} and the references therein.


{\bf Earlier results on wave breaking for the Whitham equation.} Following the methods in \cite{r12,r2,r4},  assume that $K(x)$ is a regular (continuous and integrable over $\mathbb{R}$) and positive symmetric kernel, which decreases monotonically over $\mathbb{R}^+$. Differentiating Whitham equation \eqref{w} with respect to $x$ and evaluating the resulting equation at $\xi_1(t)$ and $\xi_2(t)$, the wave breaking  can be derived from the following Riccati-type differential inequalities
\begin{align}\label{riccati}
\begin{cases}
m^{\prime}(t)\leq-m^{2}(t)+K(0)\left(M(0)-m(0)\right),\\
 M^{\prime}(t)\leq-M^{2}(t)+K(0)\left(M(0)-m(0)\right),
 \end{cases}
 \text{a.e.}~\text{on}~(0,T).
\end{align}
Constantin and Escher \cite{r4} showed that if the $M(0)$ and $m(0)$ satifies
\begin{align}\label{sdys}
m(0)+M(0)\leq -2K(0),
\end{align}
then wave breaking occurs. Later, by analyzing the monotonicity and continuity properties of  \eqref{riccati}, Ma, Liu and Qu \cite{r2} obtained an improved wave breaking criterion, namely
\begin{align}\label{Ma}
m(0)&<\min\left\{-2K(0),\,\frac{-K(0)-\sqrt{K^2(0)+4K(0)M(0)}}2
\right\}.
\end{align}


{\bf Earlier results on wave breaking for the FW equation.}
Differentiating the FW equation \eqref{fw} with respect to $x$ and evaluating the resulting equation at $\xi_1(t)$ and $\xi_2(t)$, the wave breaking  can be derived from the following Riccati-type differential inequalities
\begin{align}
\begin{cases}
 m^{\prime}(t)\leq-\frac{3}{2}m^{2}(t)+\frac{1}{2}\left(M(t)-m(t)\right),\\
 M^{\prime}(t)\leq-\frac{3}{2}M^{2}(t)+\frac{1}{2}\left(M(t)-m(t)\right).
\end{cases}
\end{align}
Constantin and Escher first proved wave breaking for the FW equation on the real line \cite{r4}.
A similar result has been shown to hold also for the periodic case in \cite{r19}, where it was shown that wave breaking occurs provided  $M(0)$ and $m(0)$ satisfy
\begin{align}\label{sdys1}
m(0)+M(0)< -\frac{2}{3}.
\end{align}
Haziot \cite{r20} provided if  $M(0)$ and $m(0)$ satisfy
\begin{align}\label{hatj}
5k m(0)+kM(0)\leq -4,\qquad k\in\left(0,\frac{3}{5}\right),
\end{align}
then  wave breaking occurs.  More recently, Wei \cite{r3} derived an improved wave breaking criterion, namely
\begin{align}\label{fwys}
m(0)<\min\left\{-\frac{2}{3},\,-\frac{1}{6}\left(1+\sqrt{1+12M(0)}\right)\right\}.
\end{align}


The main purpose of this paper is to rigorously derive  more general wave breaking criteria that strictly extends earlier wave breaking criteria for the Whitham equation and the FW equation. The proof relies on an analysis of the monotonicity and continuity properties of  $m(t)+\lambda M(t) $ with $\lambda>0$.  Our wave breaking criterion for the Whitham equation read
\begin{align*}
m(0)<-\lambda M(0)-\frac{(1+\lambda)^2K(0)}{2\sqrt{\lambda}},~~~\lambda>0,
\end{align*}
which covers wave breaking criteria \eqref{sdys}-\eqref{Ma}(see Remark \ref{rek1} for details).  For the FW equation,  wave breaking criterion is given by
\begin{align*}
m(0)<-\lambda M(0)-\frac{(1+\lambda)^2}{6\sqrt{\lambda}},
\end{align*}
which covers wave breaking criteria \eqref{sdys1}-\eqref{fwys}(see Remark \ref{rek2} for details). Furthermore, we present further generalizations of our approach to other nonlocal equations, including the Camassa-Holm (CH) equation, the $\mu$-Camassa-Holm ($\mu$-CH) equation and the hyperelastic rod equation.


The remainder of this paper is organized as follows.  In Section \ref{sec2}, we present our main results.
In Section \ref{sec3} and Section \ref{sec4}, we prove our wave breaking results.

\section{The main results}\label{sec2}
In this section, we present  the main results concerning  wave breaking for the Whitham equation  and the FW equation.

Our main result concerning  wave breaking for the Whitham equation \eqref{w}  is as follows.
\begin{theorem}\label{wgj}
If $K(x)$ in \eqref{w} is a regular (continuous and integrable over $\mathbb{R}$) and positive symmetric kernel, which decreases monotonically over $\mathbb{R}_+$. Let $u(t,x)$ be a solution of the Cauchy problem \eqref{w} with initial data $u_0\in H^\infty(\mathbb{R})$ and $T$ be the maximal existence time of the unique solution $u$ of \eqref{w}. $m(t)$ and $M(t)$ are given in \eqref{a.2}.  For any $\lambda>0$, assume that $m(0)$ and $M(0)$ satisfy
\begin{align}\label{w0}
m(0)<-\lambda M(0)-\frac{(1+\lambda)^2K(0)}{2\sqrt{\lambda}},
\end{align}
then the corresponding solution $u$ breaks in finite time with an estimate of the blow-up time $T$
as
\begin{align*}
0<T\leq&\frac{\sqrt{\lambda}}{K(0)(1+\lambda)}\ln\frac{-2\sqrt{\lambda}\left(m(0)+\lambda M(0)\right)+K(0)(1+\lambda)^2}{-2\sqrt{\lambda}\left(m(0)+\lambda M(0)\right)-K(0)(1+\lambda)^2}.
\end{align*}
\end{theorem}
\begin{remark}\label{rek1}
\normalfont
~\\~\noindent{\rm(1)}
When $\lambda=1$, the wave breaking criterion \eqref{w0} reduces to Constantin-Escher's wave breaking criterion \eqref{sdys}.\\

\noindent{\rm(2)}
Let
\begin{align*}
F_{M_0}(\lambda):=\lambda M(0)+\frac{(1+\lambda)^2K(0)}{2\sqrt{\lambda}},
\end{align*}
For some $M_0\geq0$, a direct computation gives
\begin{align*}
F_{M_0}'(\lambda)=M_0+\frac{K(0)(3\lambda^2+2\lambda-1)}{4\lambda^{3/2}}\quad\text{and}\quad
F_{M_0}''(\lambda)=\frac{K(0)(3\lambda^2-2\lambda+3)}{8\lambda^{5/2}}>0.
\end{align*}
Hence $F_{M_0}(\lambda)$ is strictly convex and has a unique minimizer $\lambda_*=\lambda_*(M_0)>0$, which is determined by
\begin{align}\label{Bb1}
4M_0\lambda_*^{3/2} + K(0)(3\lambda_*^2+2\lambda_*-1)=0.
\end{align}

For $0\leq M_0\leq 2K(0)$, we have
\begin{align*}
\frac{d}{dM_0}\left(\lambda_*M_0+\frac{(1+\lambda_*)^2K(0)}{2\sqrt{\lambda_*}}\right)=
\lambda_*+\left(M_0+\frac{K(0)(3\lambda_*^2+2\lambda_*-1)}
{4\lambda_*^{3/2}}\right)\frac{d\lambda_*}{dM_0}
=\lambda_*>0,
\end{align*}
then $F_{M_0}(\lambda_*)$ is increasing with respect to $M_0$. When $M_0=2K(0)$,   $\lambda_*$ satisfy
\begin{align*}
8\lambda_*^{3/2}+3\lambda_*^2+2\lambda_*-1=0,
\end{align*}
which implies
\begin{align*}
\lambda_*=3-2\sqrt2.
\end{align*}
Therefore, one has
\begin{align*}
F_{M_0}(\lambda_*)\leq2K(0)(3-2\sqrt2)+
\frac{(4-2\sqrt2)^2K(0)}{2(\sqrt2-1)}=2K(0).
\end{align*}
Consequently, for $0\leq M_0\leq 2K(0)$, we have
\begin{align}\label{Bb2}
\min\left\{-2K(0),\,\frac{-K(0)-\sqrt{K^2(0)+4K(0)M_0}}2
\right\}=-2K(0)\leq-F_{M_0}(\lambda_*).
\end{align}

For $M_0>2K(0)$, let $r_*=\sqrt{\lambda_*}$, then $0<r_*<\sqrt2-1$.  Together with Eq.\eqref{Bb1}, one has
\begin{align*}
M_0=\frac{K(0)(1-2r_*^2-3r_*^4)}{4r_*^3},\qquad F_{M_0}(\lambda_*)
=\frac{K(0)(3+2r_*^2-r_*^4)}{4r_*}.
\end{align*}
A direct calculation yields
\begin{align*}
F_{M_0}^2(\lambda_*)-K(0)F_{M_0}(\lambda_*)-K(0)M_0=\frac{K^2(0)(1+r_*^2)^2(r_*^2+2r_*-1)(r_*^3-2r_*^2-r_*+4)}{16r_*^3}<0.
\end{align*}
Since
\begin{align*}
\frac{K(0)+\sqrt{K^2(0)+4K(0)M_0}}2
\end{align*}
is the positive root of \(z^2-K(0)z-K(0)M_0=0\), it follows that
\begin{align*}
F_{M_0}(\lambda_*)<\frac{K(0)+\sqrt{K^2(0)+4K(0)M_0}}2,
\qquad M_0>2K(0).
\end{align*}
Consequently, for $M_0>2K(0)$, we have
\begin{align}\label{Bb3}
m_0<\min\left\{-2K(0),\,\frac{-K(0)-\sqrt{K^2(0)+4K(0)M_0}}2\right\}
\leq-F_{M_0}(\lambda_*).
\end{align}
Combining \eqref{Bb2} and \eqref{Bb3},  we find that our wave breaking criterion \eqref{w0} covers wave breaking criterion \eqref{Ma}.
\end{remark}
\begin{figure}[H]
\centering

\captionsetup{
    font=small,
    skip=5pt
}

\begin{adjustbox}{max width=0.98\linewidth,center}

\begin{tikzpicture}

\begin{axis}[
    width=9cm,
    height=6.8cm,
    xmin=-9.6,
    xmax=0.6,
    ymin=0,
    ymax=16.2,
    axis x line=middle,
    axis y line=middle,
    axis line style={draw=none},
    tick style={black},
    xtick={-9,-7.5,-6,-4.5,-3,-1.5},
    ytick={3,6,9,12,15},
    xticklabel style={
        font=\small,
        anchor=north
    },
    yticklabel style={
        font=\small,
        anchor=east
    },
    clip=false,
    samples=300,
    legend style={
        at={(1.025,0.50)},
        anchor=west,
        font=\scriptsize,
        draw=gray,
        fill=white,
        cells={anchor=west},
        row sep=14pt,
        inner xsep=5pt,
        inner ysep=8pt
    }
]

%
%
%

\addplot[
    very thick,
    solid,
    red,
    variable=\s,
    domain=0.2382883086:0.5773502692,
    samples=500
]
(
    {(\s^4-2*\s^2-3)/(4*\s)},
    {(1-2*\s^2-3*\s^4)/(4*\s^3)}
);

\addlegendentry{
    $\displaystyle
    \widetilde m_0=
    \sup_{\lambda>0}
    \left\{
    -\lambda\widetilde M_0
    -\frac{(1+\lambda)^2}{2\sqrt{\lambda}}
    \right\}$
}

%

\addplot[
    very thick,
    solid,
    blue
]
coordinates {
    (-2,0)
    (-2,2)
};

\addlegendentry{
    $\displaystyle
    \widetilde m_0=
    \min\left\{
    -2,\,
    \frac{-1-\sqrt{1+4\widetilde M_0}}{2}
    \right\}$
}

%
%
%

\addplot[
    very thick,
    solid,
    blue,
    domain=-4.5558599581:-2,
    samples=500,
    forget plot
]
{x^2+x};

%
%

\addplot[
    very thick,
    solid,
    green!50!black,
    domain=-9.6:-2
]
{-2-x};

\addlegendentry{
    $\displaystyle
    \widetilde m_0+\widetilde M_0=-2$
}


\addplot[
    only marks,
    mark=*,
    mark size=2pt,
    black,
    forget plot
]
coordinates {
    (-2,2)
};

\node[
    anchor=south east,
    font=\scriptsize
]
at (axis cs:-2.25,2.25)
{$(-2,2)$};


\draw[->,thick]
    (axis cs:-9.6,0)
    --
    (axis cs:0.72,0);

\node[
    anchor=north east,
    font=\scriptsize
]
at (axis description cs:1,-0.10)
{$
    \widetilde m_0=\frac{m(0)}{K(0)}
$};

\draw[->,thick]
    (axis cs:0,0)
    --
    (axis cs:0,16.55);

\node[
    anchor=south,
    font=\scriptsize
]
at (axis description cs:0.941,1.03)
{$
    \widetilde M_0=\frac{M(0)}{K(0)}
$};


\node[
    below right,
    font=\small
]
at (axis cs:0,0)
{$0$};

\end{axis}

\end{tikzpicture}

\end{adjustbox}

\caption{
The region to the left of the red curve represents the optimized wave breaking criterion stated in  Theorem \ref{wgj}. The region to the left of the blue curve represents the wave breaking criterion stated in \eqref{Ma}, whereas the region lying to the left of the green line indicates the wave breaking criterion given in \eqref{sdys}.
}
\label{fig1}

\end{figure}
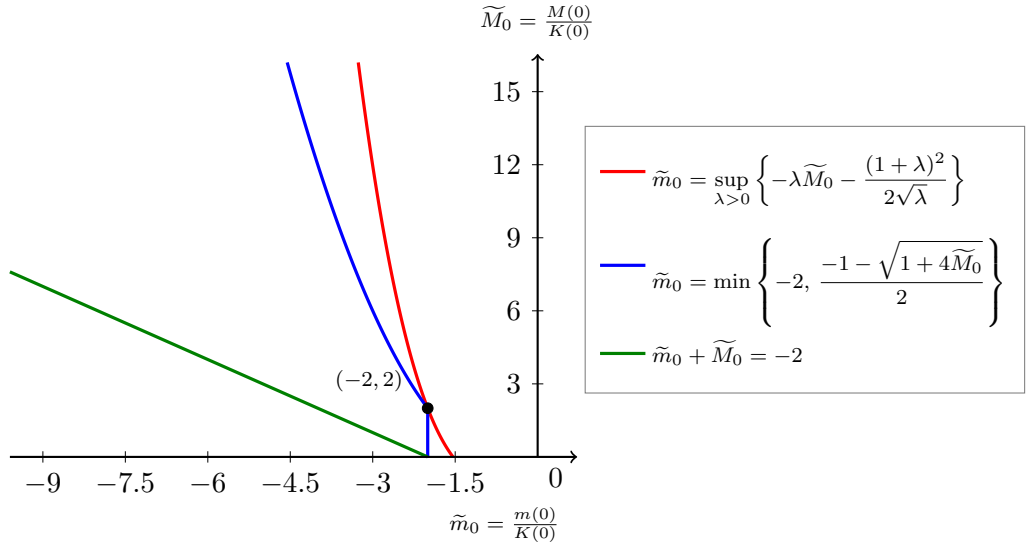

Our main result concerning   wave breaking for the  FW equation \eqref{fw}  is as follows.

\begin{theorem}\label{fwgj}
Let the initial data $u_0\in H^3(\mathbb{R})$. Let $T$ be the maximal existence time of the unique solution $u$ of  the Cauchy problem  \eqref{fw}. $m(t)$ and $M(t)$ are given in \eqref{a.2}.  For any $\lambda>0$, assume that $m(0)$ and $M(0)$ satisfy
\begin{align}\label{fw0}
m(0)<-\lambda M(0)-\frac{(1+\lambda)^2}{6\sqrt{\lambda}},
\end{align}
then the corresponding solution $u$ breaks in finite time with an estimate of the blow-up time $T$ as
\begin{align*}
0<T\leq\frac{2\sqrt{\lambda}}{1+\lambda}
\ln\frac{-6\sqrt{\lambda}\bigl(m(0)+\lambda M(0)\bigr)+(1+\lambda)^2}{-6\sqrt{\lambda}\bigl(m(0)+\lambda M(0)\bigr)-(1+\lambda)^2}.
\end{align*}
\end{theorem}
\begin{remark}\label{rek2}
\normalfont
~\\~\noindent{\rm(1)}
When $\lambda=1$,  wave breaking criterion  \eqref{fw0} reduces to  wave breaking criterion \eqref{sdys1}.\\

\noindent{\rm(2)} Let $\lambda=\frac{1}{5}$, note that $M(0)\geq 0$, it is not difficult to check that
\begin{align*}
m(0)<-\frac{1}{5}M(0)-\frac{4}{5k}<-\frac{1}{5}M(0)-\frac{6\sqrt{6}}{25},~~~k\in\left(0, \frac{3}{5} \right).
\end{align*}
Hence,  our wave breaking criterion \eqref{fw0} covers wave breaking criterion \eqref{hatj}.\\


\noindent{\rm(3)} Let
\begin{align*}
G_{M_0}(\lambda):=\lambda M_0+\frac{(1+\lambda)^2}{6\sqrt{\lambda}},
\end{align*}
For fixed $M_0\geq0$, a direct computation gives
\begin{align*}
G_{M_0}'(\lambda)=M_0+\frac{3\lambda^2+2\lambda-1}{12\lambda^{3/2}},\qquad G_{M_0}''(\lambda)=\frac{3\lambda^2-2\lambda+3}{24\lambda^{5/2}}>0.
\end{align*}
Hence $G_{M_0}$ is strictly convex and has a unique minimizer
$\lambda_*=\lambda_*(M_0)>0$, which is determined by
\begin{align}\label{Ac1}
12M_0\lambda_*^{3/2}+3\lambda_*^2+2\lambda_*-1=0.
\end{align}

For $0\leq M_0\leq 2/3$, we have
\begin{align*}
\frac{d}{dM_0}
\left(\lambda_*M_0+\frac{(1+\lambda_*)^2}{6\sqrt{\lambda_*}}
\right)=\lambda_*+\left(M_0+\frac{3\lambda_*^2+2\lambda_*-1}{12\lambda_*^{3/2}}
\right)\frac{d\lambda_*}{dM_0}=\lambda_*>0,
\end{align*}
which means $G_{M_0}(\lambda_*)$ is increasing with respect to $M_0$, when $M_0=2/3$,  $\lambda_*$ satisfy
\begin{align*}
8\lambda_*^{3/2}+3\lambda_*^2+2\lambda_*-1=0,
\end{align*}
which gives
\begin{align*}
\lambda_*=3-2\sqrt{2}.
\end{align*}
Therefore, one has
\begin{align*}
-\lambda_*M_0-\frac{(1+\lambda_*)^2}{6\sqrt{\lambda_*}}
=-\frac{2}{3}(3-2\sqrt{2})-\frac{\left(4-2\sqrt{2}\right)^2}{6(\sqrt2-1)}=-\frac{2}{3}=\min\left\{-\frac{2}{3},\,\frac{-1-\sqrt{1+12M_0}}6\right\}.
\end{align*}
Consequently, for $0\leq M_0\leq 2/3$, we have
\begin{align}\label{Ad1}
\min\left\{-\frac{2}{3},\,-\frac{1+\sqrt{1+12M_0}}6\right\}=-\frac{2}{3}
\leq-\lambda_*M_0-\frac{(1+\lambda_*)^2}{6\sqrt{\lambda_*}}.
\end{align}

For $M_0>2/3$, let $r_*=\sqrt{\lambda_*}$, then $0<r_*<\sqrt2-1$.  Together with \eqref{Ac1}, one has
\begin{align*}
M_0=\frac{1-2r_*^2-3r_*^4}{12r_*^3},\quad
\lambda_*M_0+\frac{(1+\lambda_*)^2}{6\sqrt{\lambda_*}}=\frac{3+2r_*^2-r_*^4}{12r_*}.
\end{align*}
A direct calculation yields
\begin{align*}
3\left(\frac{3+2r_*^2-r_*^4}{12r_*}\right)^2-\frac{3+2r_*^2-r_*^4}{12r_*}-M_0=\frac{(1+r_*^2)^2(r_*^2+2r_*-1)(r_*^3-2r_*^2-r_*+4)}{48r_*^3}<0.
\end{align*}
Since $(1+\sqrt{1+12M_0})/6$ is the positive root of $3z^2-z-M_0=0$, it follows that
\begin{align*}
\lambda_*M_0+\frac{(1+\lambda_*)^2}{6\sqrt{\lambda_*}}<\frac{1+\sqrt{1+12M_0}}{6},\qquad M_0>\frac{2}{3}.
\end{align*}
Consequently, for $M_0>2/3$, we have
\begin{align}\label{Ad2}
m_0<\min\left\{-\frac{2}{3},\,\frac{-1-\sqrt{1+12M_0}}6\right\}
\leq-\lambda_*M_0-\frac{(1+\lambda_*)^2}{6\sqrt{\lambda_*}}.
\end{align}
Combining \eqref{Ad1} and \eqref{Ad2},  we find that our wave breaking criterion \eqref{fw0} covers wave breaking criterion  \eqref{fwys}.
\end{remark}
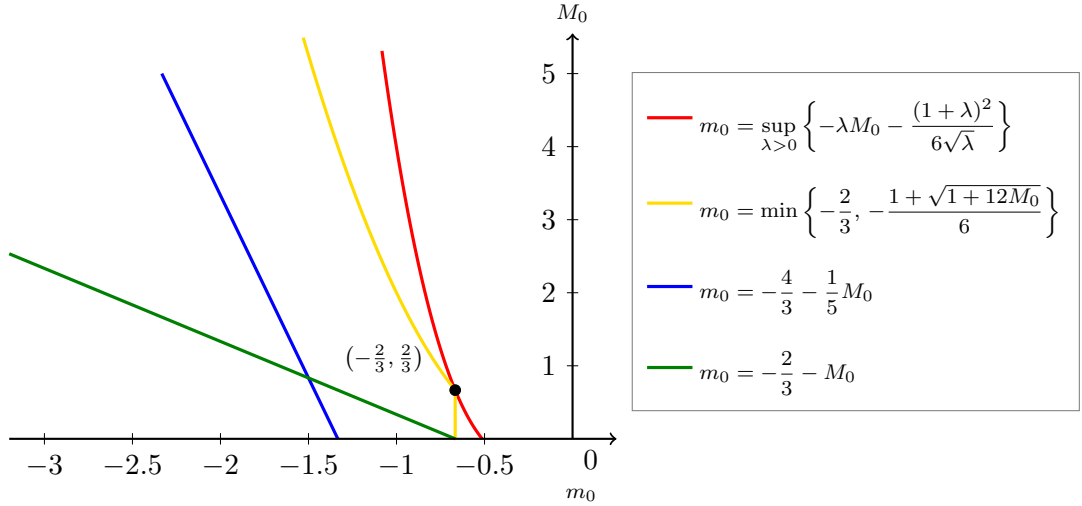
\begin{figure}[H]
\centering

\captionsetup{
    font=small,
    skip=5pt
}

\begin{tikzpicture}

\begin{axis}[
    width=9.5cm,
    height=6.8cm,
    xmin=-3.2,
    xmax=0.2,
    ymin=0,
    ymax=5.4,
    axis x line=middle,
    axis y line=middle,
    axis line style={draw=none},
    tick style={black},
    xtick={-3,-2.5,-2,-1.5,-1,-0.5},
    ytick={1,2,3,4,5},
    xticklabel style={
        font=\small,
        anchor=north
    },
    yticklabel style={
        font=\small,
        anchor=east
    },
    clip=false,
    samples=300,
    legend style={
        at={(1.04,0.50)},
        anchor=west,
        font=\scriptsize,
        draw=gray,
        fill=white,
        cells={anchor=west},
        row sep=14pt,
        inner xsep=5pt,
        inner ysep=8pt
    }
]

%
%

\addplot[
    very thick,
    solid,
    red,
    variable=\s,
    domain=0.2395:0.5773502692,
    samples=500
]
(
    {(\s^4-2*\s^2-3)/(12*\s)},
    {(1-2*\s^2-3*\s^4)/(12*\s^3)}
);

\addlegendentry{
    $\displaystyle
    m_0=
    \sup_{\lambda>0}
    \left\{
    -\lambda M_0
    -\frac{(1+\lambda)^2}{6\sqrt{\lambda}}
    \right\}$
}


\addplot[
    very thick,
    solid,
    yellow!80!orange
]
coordinates {
    (-2/3,0)
    (-2/3,2/3)
};

\addplot[
    very thick,
    solid,
    yellow!80!orange,
    domain=-1.530: -0.666666666667,
    samples=500,
    forget plot
]
{3*x^2+x};

\addlegendentry{
    $\displaystyle
    m_0=
    \min\left\{
    -\frac{2}{3},\,
    -\frac{1+\sqrt{1+12M_0}}{6}
    \right\}$
}


\addplot[
    very thick,
    solid,
    blue,
    domain=-7/3:-4/3
]
{-5*x-20/3};

\addlegendentry{
    $\displaystyle
    m_0=-\frac{4}{3}-\frac15 M_0
    $
}


\addplot[
    very thick,
    solid,
    green!50!black,
    domain=-3.2:-2/3
]
{-2/3-x};

\addlegendentry{
    $\displaystyle
    m_0=-\frac23-M_0
    $
}


\addplot[
    only marks,
    mark=*,
    mark size=2pt,
    black,
    forget plot
]
coordinates {
    (-2/3,2/3)
};

\node[
    anchor=south east,
    font=\scriptsize
]
at (axis cs:-0.78,0.76)
{$\left(-\frac23,\frac23\right)$};


\draw[->,thick]
    (axis cs:-3.2,0)
    --
    (axis cs:0.25,0);

\node[
    anchor=north east,
    font=\scriptsize
]
at (axis description cs:1,-0.10)
{$m_0$};

\draw[->,thick]
    (axis cs:0,0)
    --
    (axis cs:0,5.55);

\node[
    anchor=south,
    font=\scriptsize
]
at (axis description cs:0.941,1.03)
{$M_0$};


\node[
    below right,
    font=\small
]
at (axis cs:0,0)
{$0$};

\end{axis}

\end{tikzpicture}

\caption{The region to the left of the red curve represents the optimized wave breaking criterion stated in Theorem \ref{fwgj}. The region to the left of the yellow curve represents Wei's wave breaking criterion \eqref{fwys}. The region lying to the left of the blue line depicts Haziot's wave breaking criterion \eqref{hatj}, whereas that bounded to the left by the green line indicates Constanti-Escher's wave breaking criterion \eqref{sdys1}.
}
\label{fig2}

\end{figure}

\section{Poof of Theorem \ref{wgj}}\label{sec3}

\begin{proof}[Proof of Theorem~\ref{wgj}]
Differentiating \eqref{w} with respect to $x$ and evaluating the resulting equation respectively at $x=\xi_1(t)$ and $\xi_2(t)$, we get
that, for a.e. $t\in(0,T)$,
\begin{align}\label{w1}
m^{\prime}(t)+m^{2}(t)+\int_{\mathbb{R}}K(\eta)u_{xx}(t,\xi_1(t)-\eta)d\eta=0
\end{align}
and
\begin{align}\label{w2}
M^{\prime}(t)+M^{2}(t)+\int_{\mathbb{R}}K(\eta)u_{xx}(t,\xi_2(t)-\eta)d\eta=0.
\end{align}
Applying the same argument in \cite{r4}, we have
\begin{align*}
\int_{\mathbb{R}}K(\eta)u_{xx}(t,\xi_i(t)-\eta)d\eta\leq K(0)(M(t)-m(t)),\quad i=1,2.
\end{align*}
and from \eqref{w1} and \eqref{w2}, we find that, for a.e. $t\in(0,T)$,
\begin{align}\label{w3}
m^{\prime}(t)\leq-m^{2}+K(0)\left(M-m\right),
\end{align}
\begin{align}\label{w4}
 M^{\prime}(t)\leq-M^{2}+K(0)\left(M-m\right).
\end{align}
For any $\lambda>0$, let
\begin{align*}
h_{\lambda}(t)=m(t)+\lambda M(t),
\end{align*}
then from \eqref{w3} and \eqref{w4}, simple computations yield that
\begin{align}\label{w5}
h'_{\lambda}(t)&=m'(t)+\lambda M'(t)\nonumber\\
&\leq -m^{2}+K(0)\left(M-m\right)+\lambda\left[-M^{2}+K(0)\left(M-m\right)\right]\nonumber\\
&=-\left(m^2+\lambda M^2\right)+K(0)(1+\lambda)\left(M-m\right)\nonumber\\
&\leq -\left[\left(h_{\lambda}(t)-\lambda M\right)^2+\lambda M^2\right]
+K(0)(1+\lambda)\left[-h_{\lambda}(t)+\left(1+\lambda\right)M\right]\nonumber\\
&=-\frac{1}{1+\lambda}\left(h^2_{\lambda}(t)-\frac{K^2(0)(1+\lambda)^4}{4\lambda}\right)
-\lambda(1+\lambda)\left(M-\frac{K(0)(1+\lambda)}{2\lambda}-\frac{h_{\lambda}(t)}{1+\lambda}\right)^2\nonumber\\
&\leq -\frac{1}{1+\lambda}\left(h^2_{\lambda}(t)-\frac{K^2(0)(1+\lambda)^4}{4\lambda}\right)\nonumber\\
&=:-\frac{1}{1+\lambda}\left(h^2_{\lambda}(t)-C_1^2\right),
\end{align}
where $C_1=\frac{(1+\lambda)^2K(0)}{2\sqrt{\lambda}}$. In view of the initial condition \eqref{fw0}, we have
\begin{align*}
h_{\lambda}(0)< -C_1\quad \text{and}\quad
h'_{\lambda}(0)\leq-\frac{1}{1+\lambda}\left(h^2_{\lambda}(0)-C_1^2\right)<0.
\end{align*}
Then, we claim that over the time of existence it holds
\begin{align}\label{w6}
h_\lambda(t)<-C_1,~~\text{for~all}~t\in[0,T).
\end{align}
If the above claim is not true, then there exists $t_0\in[0,T)$ such that
\begin{align}\label{w7}
t_0=\inf\left\{t\in[0,T): h_\lambda(t)=-C_1\right\}.
\end{align}
It is easy to see that $t_0>0$. The definition of $t_0$ means that
\begin{align}\label{w8}
h_{\lambda}(t_0)=-C_1\quad \text{and}\quad h_{\lambda}(t)<-C_1\quad \text{for all }t\in[0,t_0).
\end{align}
Thus, from \eqref{w5}, we have
\begin{align}
h'_{\lambda}(t)\leq -\frac{1}{1+\lambda}\left(h^2_{\lambda}(t)-C_1^2\right)<0 \quad \text{for a.e. } t\in[0,t_0),
\end{align}
which means
\begin{align}
h_{\lambda}(t_0)\leq h_{\lambda}(0)< -C_1.
\end{align}
This is contradiction with \eqref{w8}. Thus we prove \eqref{w6}.

Now, in view of \eqref{w6}, we can solve the differential inequality \eqref{w5}. Since $h_\lambda(t)<-C_1$, we have $h_\lambda^2(t)-C_1^2>0$. Thus
\begin{align*}
\frac{h'_\lambda(t)}{h_\lambda^2(t)-C_1^2}
\leq-\frac{1}{1+\lambda}.
\end{align*}
Integrating above inequality yields
\begin{align*}
h_\lambda(t)\leq-C_1\left(\frac{-h_\lambda(0)+C_1+\left(-h_\lambda(0)-C_1\right)
\exp\left(\frac{2C_1}{1+\lambda}t\right)}{-h_\lambda(0)+C_1-\left(-h_\lambda(0)-C_1\right)\exp\left(\frac{2C_1}{1+\lambda}t\right)}\right),
\end{align*}
and then
\begin{align*}
m(t)\leq \frac{m(t)+\lambda M(t)}{1+\lambda}=\frac{h_{\lambda}(t)}{1+\lambda}\rightarrow -\infty,\quad\text{as}\quad t\rightarrow\frac{1+\lambda}{2C_1}\ln\frac{-h_\lambda(0)+C_1}{-h_\lambda(0)-C_1},
\end{align*}
which implies that the solution $u$ breaks down at a finite time $T$ with
\begin{align*}
0<T\leq T^*=\frac{1+\lambda}{2C_1}\ln\frac{-m(0)-\lambda M(0)+C_1}{-m(0)-\lambda M(0)-C_1}.
\end{align*}
This completes the proof of Theorem \ref{wgj}.
\end{proof}

\section{Poof of Theorem \ref{fwgj}}\label{sec4}

\begin{proof}[Proof of Theorem~\ref{fwgj}]
Differentiating \eqref{fw} with respect to $x$ and evaluating the resulting equation respectively at $x=\xi_1(t)$ and $\xi_2(t)$, we get
that, for a.e. $t\in(0,T)$,
\begin{align}\label{fw1}
m^{\prime}(t)+\frac{3}{2}m^{2}=[G_x*u_x](t,\xi_1(t)),~~M^{\prime}(t)+\frac{3}{2}M^{2}=[G_x*u_x](t,\xi_2(t)).
\end{align}
By observing
\begin{align*}
|G_x*h_x|&=\left|-\frac{1}{2}\int_{-\infty}^x\ep^{y-x}h_xdy+\frac{1}{2}\int_x^{+\infty}\ep^{x-y}h_xdy\right|\leq\frac{1}{2}(M(t)-m(t)),
\end{align*}
it's easy to find that, for a.e. $t\in(0,T)$,
\begin{align}\label{fw2}
-\frac{3}{2}m^{2}-\frac{1}{2}\left(M-m\right)\leq m^{\prime}(t)\leq-\frac{3}{2}m^{2}+\frac{1}{2}\left(M-m\right),
\end{align}
\begin{align}\label{fw3}
-\frac{3}{2}M^{2}-\frac{1}{2}\left(M-m\right)\leq M^{\prime}(t)\leq-\frac{3}{2}M^{2}+\frac{1}{2}\left(M-m\right).
\end{align}
For any $\lambda>0$, let
\begin{align}\label{fw4}
h_{\lambda}(t)=m(t)+\lambda M(t),
\end{align}
then from \eqref{fw2} and \eqref{fw3}, simple computations yield that
\begin{align}\label{fw5}
h'_{\lambda}(t)&=m'(t)+\lambda M'(t)\nonumber\\
&\leq -\frac{3}{2}m^{2}+\frac{1}{2}\left(M-m\right)+\lambda\left[-\frac{3}{2}M^{2}+\frac{1}{2}\left(M-m\right)\right]\nonumber\\
&=-\frac{3}{2}\left(m^2+\lambda M^2\right)+\frac{1+\lambda}{2}\left(M-m\right)\nonumber\\
&\leq -\frac{3}{2}\left[\left(h_{\lambda}(t)-\lambda M\right)^2+\lambda M^2\right]+\frac{1+\lambda}{2}\left[-h_{\lambda}(t)+\left(1+\lambda\right)M\right]\nonumber\\
&=-\frac{3}{2(1+\lambda)}\left(h^2_{\lambda}(t)-\frac{(1+\lambda)^4}{36\lambda}\right)-\frac{3}{2}\lambda(1+\lambda)\left(M-\frac{1+\lambda}{6\lambda}-\frac{h_{\lambda}(t)}{1+\lambda}\right)^2\nonumber\\
&\leq -\frac{3}{2(1+\lambda)}\left(h^2_{\lambda}(t)-\frac{(1+\lambda)^4}{36\lambda}\right)\nonumber\\
&=:-\frac{3}{2(1+\lambda)}\left(h^2_{\lambda}(t)-C_2^2\right),
\end{align}
where $C_2=\frac{(1+\lambda)^2}{6\sqrt{\lambda}}$. The subsequent steps are  similar to the proof of Theorem \ref{wgj}, and hence are omitted.

\end{proof}

\section{Further generalizations}
It is noted that some  nonlinear nonlocal  models, such as the CH equation \cite{r21}
\begin{align}\label{CH}
\begin{cases}
u_t+uu_x+G_x*\left(u^2+\frac{1}{2}u_x^2 \right)=0,\\
u(0,x)=u_0(x),
\end{cases}
\end{align}
the $\mu$-CH equation \cite{r22}
 \begin{align}\label{muCH}
\begin{cases}
u_t+uu_x+G_x*\left(2\mu(u)u+\frac{1}{2}u_x^2+\gamma u \right)=0,~~~~~\mu(u)=\int_0^1u(t,x)dx,\\
u(0,x)=u_0(x),
\end{cases}
\end{align}
and the hyperelastic rod equation \cite{r23}
\begin{align}\label{hreq}
\begin{cases}
u_t+\gamma uu_x+G_x*\left(\frac{3-\gamma}{2}u^2+\frac{\gamma}{2}u_x^2+\rho u \right)=0,\\
u(0,x)=u_0(x),
\end{cases}
\end{align}
can be form  a  more general system of the Riccati-type differential inequalities as follows:
\begin{align}\label{gy0}
\begin{cases}
m'(t)\leq -a m^2(t)+b\left(M(t)-m(t)\right)+c,\\[2mm]
M'(t)\leq -a M^2(t)+b\left(M(t)-m(t)\right)+c,
\end{cases}~\text{a.e.}~\text{on}~(0,T),
\end{align}
where $a$ is a positive constant, $b$ and $c$ are non-negative constants,   and $M(t)$ is a nonnegative function of $t$.  The approach for proving Theorem \ref{wgj} and Theorem \ref{fwgj}  can be extended, with suitable modifications, to \eqref{gy0}, thereby yielding the following wave breaking result.

\begin{theorem}\label{gygj}
Let $m(t)$ and $M(t)$ be two continuous and almost everywhere differentiable functions defined in $t\in[0,T)$ with $T<\infty$ satisfying \eqref{gy0}. For any $\lambda>0$, suppose that the initial data $m(0)$ and $M(0)$ satisfy
\begin{align}\label{gy1}
m(0)<-\lambda M(0)-\dfrac{1+\lambda}{a}\sqrt{ac+\frac{b^2(1+\lambda)^2}{4\lambda}}.
\end{align}
Then the corresponding solution $u$ breaks in finite time with an estimate of the blow-up time $T$ as
\begin{align*}
0<T\leq\frac{1}{
2\sqrt{ac+\frac{b^2(1+\lambda)^2}{4\lambda}}}\ln
\frac{-a\bigl(m(0)+\lambda M(0)\bigr)+(1+\lambda)\sqrt{ac+\frac{b^2(1+\lambda)^2}{4\lambda}}}{-a\bigl(m(0)+\lambda M(0)\bigr)-(1+\lambda)\sqrt{ac+\frac{b^2(1+\lambda)^2}{4\lambda}}}.
\end{align*}
\end{theorem}
\begin{remark}
\normalfont
Let
\begin{align*}
A:=\frac{b+\sqrt{b^2+ac}}{a},\qquad
B_{M_0}:=\frac{b+\sqrt{b^2+4a(bM_0+c)}}{2a}.
\end{align*}
Then the wave breaking condition in Theorem 2.1 of \cite{r2} can be written as
\begin{align}\label{mafei}
m_0<\min\{-A,-B_{M_0}\}.
\end{align}

For $\lambda>0$, define
\begin{align*}
\Gamma_{M_0}(\lambda)
:=\lambda M_0+\frac{1+\lambda}{a}
\sqrt{ac+\frac{b^2(1+\lambda)^2}{4\lambda}}.
\end{align*}

~\\~\noindent{\rm(1)} When $b>0$, for fixed $M_0\geq0$, a direct computation gives
\begin{align*}
\Gamma_{M_0}'(\lambda)=M_0+\frac{8ac\lambda^2+b^2(1+\lambda)^2(3\lambda-1)}{8a\lambda^2\sqrt{ac+\dfrac{b^2(1+\lambda)^2}{4\lambda}}}
\end{align*}
and 
\begin{align*}
\Gamma_{M_0}''(\lambda)&=\frac{b^2(1+\lambda)\left[16ac\lambda(\lambda^2-\lambda+1)+b^2\left(3\lambda^4+4\lambda^3+2\lambda^2+4\lambda+3\right)\right]}{64a\lambda^4\left(ac+\dfrac{b^2(1+\lambda)^2}{4\lambda}\right)^{3/2}}>0.
\end{align*}
Hence $\Gamma_{M_0}$ is strictly convex and has a unique
minimizer $\lambda_*=\lambda_*(M_0)>0$, which is determined by
\begin{align}\label{Cc1}
&8aM_0\lambda_*^2
\sqrt{ac+\frac{b^2(1+\lambda_*)^2}{4\lambda_*}}
+8ac\lambda_*^2
+b^2(1+\lambda_*)^2(3\lambda_*-1)=0.
\end{align}
For $0\leq M_0\leq A$, we have
\begin{align*}
\frac{d}{dM_0}\left[
\lambda_*M_0+\frac{1+\lambda_*}{a}
\sqrt{ac+\frac{b^2(1+\lambda_*)^2}{4\lambda_*}}
\right]
&=\lambda_*+\left[M_0+\frac{8ac\lambda_*^2+b^2(1+\lambda_*)^2(3\lambda_*-1)}{8a\lambda_*^2\sqrt{ac+\dfrac{b^2(1+\lambda_*)^2}{4\lambda_*}}}\right]
\frac{d\lambda_*}{dM_0}\\
&=\lambda_*>0,
\end{align*}
then $\Gamma_{M_0}(\lambda_*)$ is increasing with respect to
$M_0$. When $M_0=A$, $\lambda_*$ satisfies
\begin{align*}
&8\left(b+\sqrt{b^2+ac}\right)\lambda_*^2
\sqrt{ac+\frac{b^2(1+\lambda_*)^2}{4\lambda_*}}
+8ac\lambda_*^2+b^2(1+\lambda_*)^2(3\lambda_*-1)=0.
\end{align*}
which implies
\begin{align*}
\lambda_*(A)
=\left(\sqrt{\frac{b+\sqrt{b^2+ac}}{b}}-\sqrt{\frac{\sqrt{b^2+ac}}{b}}\right)^2
=\left(\sqrt{\frac{aA}{b}}-\sqrt{\frac{aA}{b}-1}
\right)^2.
\end{align*}
Therefore, one has
\begin{align*}
\Gamma_{M_0}\bigl(\lambda_*(M_0)\bigr)
\leq\Gamma_A\bigl(\lambda_*(A)\bigr)&=\lambda_*(A)A+\frac{1+\lambda_*(A)}{a}\sqrt{ac+\frac{b^2\left(1+\lambda_*(A)\right)^2}{4\lambda_*(A)}}=A.
\end{align*}
Consequently, for $0\leq M_0\leq A$, we have
\begin{align}\label{Cc2}
\min\{-A,-B_{M_0}\}
=-A\leq-\Gamma_{M_0}\bigl(\lambda_*(M_0)\bigr).
\end{align}

For $M_0>A$, let $r_*=\sqrt{\lambda_*}$, then
\begin{align*}
0<r_*<\sqrt{\frac{b+\sqrt{b^2+ac}}{b}}-\sqrt{\frac{\sqrt{b^2+ac}}{b}}.
\end{align*}
Together with \eqref{Cc1}, one has
\begin{align*}
M_0&=\frac{b^2(1+r_*^2)^2(1-3r_*^2)-8ac\,r_*^4}{4ar_*^3
\sqrt{b^2(1+r_*^2)^2+4ac\,r_*^2}},
\end{align*}
\begin{align*}
\lambda_*M_0+\frac{1+\lambda_*}{a}
\sqrt{ac+\frac{b^2(1+\lambda_*)^2}{4\lambda_*}}=\frac{b^2(3+5r_*^2+r_*^4-r_*^6)+8ac\,r_*^2}{4ar_*\sqrt{b^2(1+r_*^2)^2+4ac\,r_*^2}}.
\end{align*}
A direct calculation yields
\begin{align*}
&a\left(\frac{b^2(3+5r_*^2+r_*^4-r_*^6)+8ac\,r_*^2}{
4ar_*\sqrt{b^2(1+r_*^2)^2+4ac\,r_*^2}}\right)^2
-b\left(\frac{b^2(3+5r_*^2+r_*^4-r_*^6)+8ac\,r_*^2
}{4ar_*\sqrt{b^2(1+r_*^2)^2+4ac\,r_*^2}}
\right)-bM_0-c
\\
&=\frac{b^2(1+r_*^2)^2}{16ar_*^3
\left[b^2(1+r_*^2)^2+4ac\,r_*^2\right]}
\\
&\quad\times\frac{\left[
16ac\,r_*^4-b^2\left(r_*^8-4r_*^6-10r_*^4-4r_*^2+1\right)\right]}{b^2r_*(r_*^2-3)^2(1+r_*^2)^2
+16ac\,r_*^3(2-r_*^2)+4b(1-r_*^4)\sqrt{b^2(1+r_*^2)^2+4ac\,r_*^2}}
\\
&\quad\times\left[16ac\,r_*^2(r_*^2-2)^2+b^2\left(-r_*^{10}+4r_*^8-6r_*^6-12r_*^4+15r_*^2+16\right)\right]
<0.
\end{align*}
Since $B_{M_0}$ is the positive root of $az^2-bz-(bM_0+c)=0$, it follows that
\begin{align*}
\lambda_*M_0+\frac{1+\lambda_*}{a}
\sqrt{ac+\frac{b^2(1+\lambda_*)^2}{4\lambda_*}}
<B_{M_0},\qquad M_0>A.
\end{align*}
Consequently, for $M_0>A$, we have
\begin{align}\label{Cc3}
m_0<\min\{-A,-B_{M_0}\}=-B_{M_0}<-\Gamma_{M_0}(\lambda_*).
\end{align}
Combining \eqref{Cc2} and \eqref{Cc3}, we find that our
wave breaking condition \eqref{gy1} covers the wave breaking condition \eqref{mafei}.\\
\noindent{\rm(2)} When $b=0$, we have
\begin{align*}
\Gamma_{M_0}(\lambda)
=\lambda M_0+(1+\lambda)\sqrt{\frac{c}{a}},
\end{align*}
then one can choose $\lambda>0$ sufficiently small such that
\begin{align*}
m_0<-\sqrt{\frac{c}{a}},
\end{align*}
which coincides with the wave breaking condition \eqref{mafei}.
\end{remark}
\pgfmathsetmacro{\sigmaVal}{1+sqrt(2)}
\pgfmathsetmacro{\minusSigmaVal}{-\sigmaVal}

\begin{figure}[H]

\centering

\captionsetup{
    font=small,
    skip=5pt
}

\begin{adjustbox}{max width=0.98\linewidth,center}

\begin{tikzpicture}

\begin{axis}[
    width=9cm,
    height=6.8cm,
    xmin=-5.2,
    xmax=0.2,
    ymin=0,
    ymax=13.2,
    axis x line=middle,
    axis y line=middle,
    axis line style={draw=none},
    tick style={black},
    xtick={-5,-4,-3,-2,-1},
    ytick={2,4,6,8,10,12},
    xticklabel style={
        font=\small,
        anchor=north
    },
    yticklabel style={
        font=\small,
        anchor=east
    },
    clip=false,
    samples=300,
    legend style={
        at={(1.025,0.50)},
        anchor=west,
        font=\scriptsize,
        draw=gray,
        fill=white,
        cells={anchor=west},
        row sep=14pt,
        inner xsep=5pt,
        inner ysep=8pt
    },
    legend cell align=left
]

%
%
%
%

\addplot[
    very thick,
    solid,
    red,
    variable=\l,
    domain=0.0594998:0.2360679775,
    samples=700
]
(
    {
        ((\l+3)*(\l^2-4*\l-1))
        /
        (
            4*sqrt(\l)*
            sqrt(\l^2+6*\l+1)
        )
    },
    {
        -((3*\l+1)*(\l^2+4*\l-1))
        /
        (
            4*\l*sqrt(\l)*
            sqrt(\l^2+6*\l+1)
        )
    }
);

\addlegendentry{
    $\displaystyle
    m_0=
    \sup_{\lambda>0}
    \left\{\,
    -\lambda M_0
    -(1+\lambda)
    \sqrt{
        1+\frac{(1+\lambda)^2}{4\lambda}
    }
    \,\right\}$
}

%

\addplot[
    very thick,
    solid,
    blue
]
coordinates {
    (\minusSigmaVal,0)
    (\minusSigmaVal,\sigmaVal)
};

\addlegendentry{
    $\displaystyle
    m_0=
    \min\left\{\,
    -1-\sqrt2,\,
    -\frac{1+\sqrt{5+4M_0}}{2}
    \,\right\}$
}

%

\addplot[
    very thick,
    solid,
    blue,
    variable=\Y,
    domain=2.4142135624:13.2,
    samples=700,
    forget plot
]
(
    {
        -(1+sqrt(5+4*\Y))/2
    },
    {\Y}
);


\addplot[
    only marks,
    mark=*,
    mark size=2pt,
    black,
    forget plot
]
coordinates {
    (\minusSigmaVal,\sigmaVal)
};

\node[
    anchor=south west,
    font=\scriptsize
]
at (axis cs:-2.33,2.54)
{$\left(-1-\sqrt2,\,1+\sqrt2\right)$};


\draw[->,thick]
    (axis cs:-5.2,0)
    --
    (axis cs:0.28,0);

\node[
    anchor=north east,
    font=\scriptsize
]
at (axis description cs:1,-0.10)
{$m_0$};

\draw[->,thick]
    (axis cs:0,0)
    --
    (axis cs:0,13.55);

\node[
    anchor=south,
    font=\scriptsize
]
at (axis description cs:0.963,1.03)
{$M_0$};


\node[
    below right,
    font=\small
]
at (axis cs:0,0)
{$0$};

\end{axis}

\end{tikzpicture}

\end{adjustbox}

\caption{
The region to the left of the red curve represents the optimized wave breaking criterion stated in Theorem \ref{gygj}.  The region to the left of the blue curve represents the wave breaking criterion given in Theorem 2.1 of \cite{r2}. (Without loss of generality, set $a=b=c=1$).
}

\label{fig3}

\end{figure}
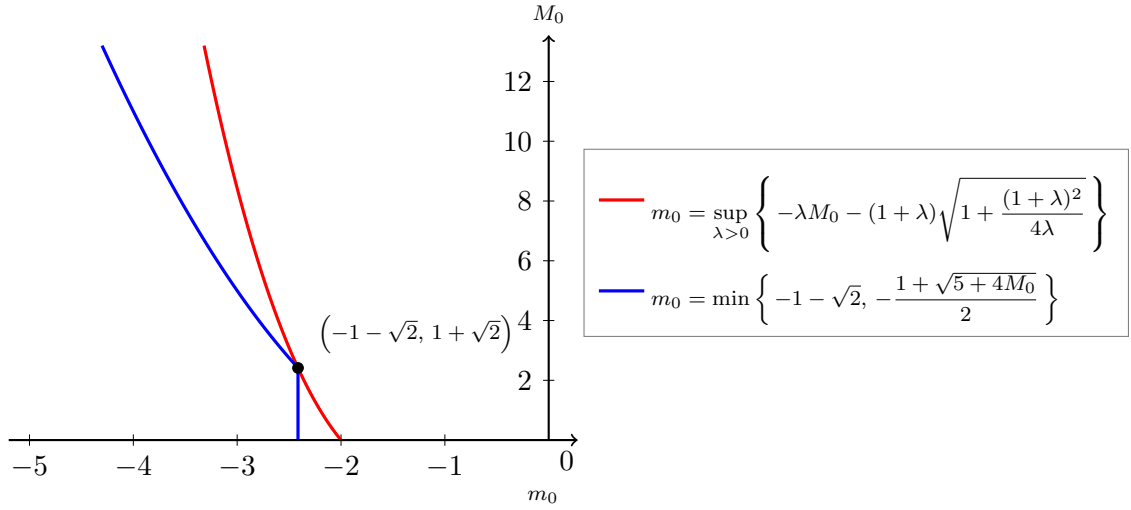
The CH equation \eqref{CH} can be form the following Riccati-type differential inequalities \cite{r31}
\begin{align}\label{}
\begin{cases}
m'(t)\leq-\dfrac{1}{2}m^2(t)
+|\gamma|\left(M(t)-m(t)\right)+C_3,\\[2mm]
M'(t)\leq-\dfrac{1}{2}M^2(t)
+|\gamma|\left(M(t)-m(t)\right)+C_3,
\end{cases}
\qquad \text{a.e. on }(0,T),
\end{align}
where
\begin{align*}
C_3=\left(\frac{1}{2}-\frac{a}{4}\right)
\left\|u_0\right\|_{H^1(\mathbb{R})}^2+\frac{\sqrt{2}}{2}|\gamma|\left\|u_0\right\|_{H^1(\mathbb{R})}-\frac{1}{4}\gamma\quad\text{with}~\gamma&=\frac{a^2+a-2}{2},~-2\leq a\leq 1.
\end{align*}
Applying Theorem \ref{gygj}, we obtain the following wave breaking result for the CH equation.
\begin{theorem}
Assume that $u\in C([0,T);H^s(\mathbb{R}))\cap C^1([0,T);H^{s-1}(\mathbb{R})), s\geq 2$ is a solution of the Cauchy problem \eqref{CH} with the initial value $u_0\in H^s(\mathbb{R})$. For any $\lambda>0$, assume that
\begin{align}
m_0<-\lambda M_0-\frac{1+\lambda}{\sqrt{\lambda}}
\sqrt{2\lambda C_3+\gamma^2(1+\lambda)^2}.
\end{align}
Then the corresponding solution of \eqref{CH} breaks down in finite time $T_0$, and
\begin{align}
T_0\leq t^*&=\frac{\sqrt{\lambda}}
{\sqrt{2\lambda C_3+\gamma^2(1+\lambda)^2}}
\ln\frac{-\sqrt{\lambda}\bigl(m_0+\lambda M_0\bigr)
+(1+\lambda)\sqrt{2\lambda C_3+\gamma^2(1+\lambda)^2}
}{-\sqrt{\lambda}\bigl(m_0+\lambda M_0\bigr)
-(1+\lambda)\sqrt{2\lambda C_3+\gamma^2(1+\lambda)^2}}.
\end{align}
\end{theorem}

The $\mu$-CH equation \eqref{muCH} can be form the following Riccati-type differential inequalities \cite{r2}
\begin{align}\label{muCHMm}
\begin{cases}
m'(t)\leq -\frac{1}{2} m^2(t)+\frac{5}{2}|\gamma|\left(M(t)-m(t)\right)+C_4,\\[2mm]
M'(t)\leq -\frac{1}{2} M^2(t)+\frac{5}{2}|\gamma|\left(M(t)-m(t)\right)+C_4,
\end{cases}\text{a.e.}~\text{on}~(0,T),
\end{align}
where
\begin{align*}
C_4=\max\left\{0,\,\frac{\sqrt{3}}{3}|\mu(u_0)|E(0)-\frac{1}{2}E^2(0)\right\},\quad\mu(u)=\mu(u_0),
\end{align*}
and
\begin{align*}
E(t)=\left(\int_0^1u_x^2dx \right)^{\frac{1}{2}}=\left(\int_0^1u_{0,x}^2dx \right)^{\frac{1}{2}}=E(0).
\end{align*}
Using Theorem \ref{gygj}, we arrive at the following wave breaking result for the $\mu$-CH equation.
\begin{theorem}
Assume that $u\in C([0,T);H^s(\mathbb{S}))
\cap C^1([0,T);H^{s-1}(\mathbb{S})), s\geq 2$,
is a solution of the Cauchy problem \eqref{muCH} with the initial value $u_0\in H^s(\mathbb{S})$. For any $\lambda>0$, suppose that
\begin{align}
m_0<-\lambda M_0
-\frac{1+\lambda}{2\sqrt{\lambda}}
\sqrt{8\lambda C_4+25\gamma^2(1+\lambda)^2}.
\end{align}
Then the corresponding solution of \eqref{muCH} blows up in finite time
\begin{align*}
T_0\leq t^*=\frac{2\sqrt{\lambda}}{\sqrt{8\lambda C_4+25\gamma^2(1+\lambda)^2}}\ln\frac{-2\sqrt{\lambda}\bigl(m_0+\lambda M_0\bigr)+(1+\lambda)\sqrt{8\lambda C_4+25\gamma^2(1+\lambda)^2}}{-2\sqrt{\lambda}\bigl(m_0+\lambda M_0\bigr)-(1+\lambda)\sqrt{8\lambda C_4+25\gamma^2(1+\lambda)^2}}.
\end{align*}
\end{theorem}

The hyperelastic rod equation \eqref{hreq} can be form the following Riccati-type differential inequalities \cite{r2}
\begin{align}\label{rodMm}
\begin{cases}
m'(t)\leq-\dfrac{\gamma}{2}m^2(t)+|\rho|\left(M(t)-m(t)\right)+C_6\left\|u_0\right\|_{H^1(\mathbb{R})}^2,
\\[2mm]
M'(t)\leq-\dfrac{\gamma}{2}M^2(t)+|\rho|\left(M(t)-m(t)\right)+C_6\left\|u_0\right\|_{H^1(\mathbb{R})}^2,
\end{cases}
\qquad
\text{a.e.}~\text{on}~(0,T),
\end{align}
where
\begin{align*}
C_5=\frac{1}{4}\left(\sqrt{\gamma(12-3\gamma)}-\gamma
\right),\quad C_6=\frac{1}{2}
\max\left\{
\frac{3-\gamma}{2}-C_5,\,0\right\},\qquad 0<\gamma<4.
\end{align*}
Making use of Theorem \ref{gygj}, we obtain the following wave breaking result for the hyperelastic rod equation.
\begin{theorem}
Assume that $0<\gamma<4$ and $u\in C([0,T);H^s(\mathbb{R}))\cap C^1([0,T);H^{s-1}(\mathbb{R})), s\geq 2$ is a solution of the Cauchy problem \eqref{hreq} with the initial value $u_0\in H^s(\mathbb{R})$. For any $\lambda>0$, suppose that
\begin{align}
m_0<-\lambda M_0-\frac{1+\lambda}{\gamma\sqrt{\lambda}}\sqrt{2\lambda\gamma C_6\left\|u_0\right\|_{H^1(\mathbb{R})}^2+\rho^2(1+\lambda)^2}.
\end{align}
Then the corresponding solution of \eqref{hreq} breaks down in finite time $T_0$ with an estimate as
\begin{align}
T_0\leq t^*=&\frac{\sqrt{\lambda}}
{\sqrt{2\lambda\gamma C_6
\left\|u_0\right\|_{H^1(\mathbb{R})}^2+\rho^2(1+\lambda)^2}}\nonumber\\
&\times\ln\frac{-\gamma\sqrt{\lambda}
\bigl(m_0+\lambda M_0\bigr)+(1+\lambda)
\sqrt{2\lambda\gamma C_6\left\|u_0\right\|_{H^1(\mathbb{R})}^2+\rho^2(1+\lambda)^2}}{-\gamma\sqrt{\lambda}\bigl(m_0+\lambda M_0\bigr)-(1+\lambda)\sqrt{2\lambda\gamma C_6\left\|u_0\right\|_{H^1(\mathbb{R})}^2+\rho^2(1+\lambda)^2}}.
\end{align}
\end{theorem}

\section*{Acknowledgment}
The authors would like to thank Professor Jean-Claude Saut for his valuable comments and suggestions, which served to improve the manuscript. \\
Yang was supported by the National Natural Science Foundation of China (No. 12561101), Yin was supported by the National Natural Science Foundation of China (No. 12571261).

\section*{Data availability }
No data was used for the research described in the article.

\section*{Conflict of interest }
The authors declare that they have no conflict of interest.

\end{document}